\documentclass{article}

\usepackage[english]{babel}

\usepackage[utf8]{inputenc}
\usepackage{stmaryrd}
\usepackage[utf8]{inputenc}
\usepackage{amsmath}
\usepackage{graphicx}
\usepackage{amssymb}
\usepackage{amsthm}
\usepackage{tikz-cd}
\usepackage{mathrsfs}
\usepackage[colorinlistoftodos]{todonotes}
\usepackage{enumitem}
\usepackage{yfonts}
\usepackage{ dsfont }
\usepackage{MnSymbol}
\usepackage{slashed}

\date{\today}
\newtheorem{thm}{Theorem}[section]
\newtheorem{lem}[thm]{Lemma}

\theoremstyle{definition}
\newtheorem{eg}[thm]{Example}

\newtheorem{rmk}{Remark}
\newtheorem{prop}[thm]{Proposition}
\newtheorem{Def}[thm]{Definition}

\newcommand{\R}{\mathbb{R}}
\newcommand{\C}{\mathbb{C}}
\newcommand{\Z}{\mathbb{Z}}

\newcommand{\norm}[1]{\left\lVert#1\right\rVert}
\newcommand{\Lap}{\Delta}

\title{From tropical curves to special Lagrangians}

\author{Shih-Kai Chiu \and Yang Li \and Yu-Shen Lin}

\date{}

\begin{document}
	
	\maketitle

\begin{abstract}
  We show that any locally planar tropical curve $\Gamma \subset \mathbb{R}^n$ (with unit edge weights) can be realized as the limit of the rescaled moment map images of a family of special Lagrangian submanifolds in $T^*T^n$ with respect to the Euclidean structure. This is based on a gluing construction that matches special Lagrangian local models to the combinatorics of $\Gamma$, thereby establishing a direct link between tropical geometry and special Lagrangian geometry.
\end{abstract}

    \section{Introduction}\label{sec: introduction}

Inside a complex $n$-dimensional Calabi-Yau manifold $(X,\omega, \Omega)$, a special Lagrangian submanifold $L\subset X$ of phase $\hat{\theta}\in \R$ is a real $n$-dimensional submanifold such that
    \begin{align*}
      \omega|_L = 0, \quad  \operatorname{Im}(e^{-i\hat\theta}\Omega)|_L=0.
    \end{align*}
 Equivalently, $L$ is calibrated by $\operatorname{Re}(e^{-i\hat\theta}\Omega)$ in the sense of Harvey–Lawson~\cite{HL}, hence $L$ is a minimal surface. The Strominger-Yau-Zaslow (SYZ) picture \cite{SYZ} suggests that a Calabi-Yau manifold $X$ near the large complex structure limit should admit a (singular) special Lagrangian $T^n$-fibration
$\pi: X \to B$,
    whose mirror $\hat X$ is obtained by fiberwise $T$-duality modulo instanton corrections. The base $B$ of the fibration admits a natural integral affine structure away from the discriminant locus, and the tropical geometry on $B$ controls both the Lagrangian geometry of $X$ (the `A-side') and the algebraic geometry of $\hat{X}$ (the `B-side').     
    
    
    In this paper, we seek to bridge tropical geometry in $B$ and \emph{special Lagrangians} in $X$, in the simplest case of $X=T^*T^n$, equipped with the Euclidean K\"ahler structure
\[
\omega=\sum d\theta_i\wedge d\mu_i, \quad \Omega= \bigwedge_1^n (d\theta_i+ \sqrt{-1} g_{ij} d\mu_j), \quad g=g_{ij} d\mu_id\mu_j+ g^{ij} d\theta_id\theta_j,
\]
where $\mu_i, \theta_i$ are the action-angle coordinates, and $(g_{ij})$ is a positive definite $n\times n$ matrix. 
In particular, the torus fibres of the natural projection $\pi: X=T^*T^n\to B:= H^1(T^n,\mathbb{R})$ are special Lagrangians of phase zero, and $\mu_i$ are affine coordinates on the base.

 Our main theorem is the following:
    
	\begin{thm}\label{thm:tropical}
      Let $n \ge 2$, $\hat{\theta}\in (0,\pi)$, and let $\Gamma \subset \mathbb{R}^n$ be the image of a locally planar tropical curve (see Def. \ref{Def:tropicalcurve}).  Then for all sufficiently large $T > 0$, there exists a special Lagrangian submanifold $ L_T \subset T^*T^n$ of phase $\hat\theta$ such that $\frac{1}{T}\pi(L_T) \longrightarrow \Gamma$ in the Hausdorff distance as $T \to +\infty$.
    \end{thm}
    
    Our result is largely motivated by the works of Matessi~\cite{Matessi21} and Mikhalkin~\cite{Mikhalkin19}, who established tropical-Lagrangian correspondence for tropical hypersurfaces $\Gamma$ (resp. tropical curves), by suitable Lagrangian smoothing $L$ of the phase lift of $\Gamma$. Our new contribution is to make $L$ into a \emph{special Lagrangian}, using a PDE based gluing construction.

    There are several open directions. 
    More generally, we expect that suitable tropical varieties $\Gamma\subset B$ of dimension $2\leq m\leq n-1$ may give rise to special Lagrangians in $T^*T^n$ whose projection to $B$ is a thickening of $\Gamma$. A local model in the case of the standard tropical hypersurface has been obtained in \cite{specialLagpairofpants}, which generalizes the pair of pants to arbitrary dimensions. In another direction, one may try to construct special Lagrangians inside compact Calabi-Yau manifolds, whose projection to the SYZ base is the thickening of a one-dimensional graph.

\begin{rmk}
On the $B$-side, the enumerative correspondence between tropical curves and holomorphic curves was initiated by Mikhalkin~\cite{Mikhalkin06} in the case of $(\C^*)^2$, and 
later generalized by Nishinou-Siebert~\cite{Nishinou-Siebert} to curves in higher dimensional toric varieties via toric degenerations. Unlike the holomorphic curve case where obstruction may arise \cite{Mikhalkin04}\cite{Speyer}, deformations of smooth special Lagrangians are unobstructed \cite{Marshall}\cite{McLean}, so obstruction conditions do not show up on the tropical curves in our setting. We further note that Hicks~\cite{Hicks} has related unobstructedness of Lagrangians in the sense of Floer homology with the realizability of tropical curves as holomorphic curves. It would be interesting to explore potential interactions between our construction and these Floer-theoretic perspectives.

  The Mikhalkin correspondence has also motivated conjectures in the special holomomy setting. In the case of $G_2$ manifolds, the role of SYZ fibrations is played by Kovalev-Lefschetz fibrations, and gradient cycles play the role of tropical curves.  Donaldson-Scaduto \cite{Donaldson-Scaduto} conjectured that the enumerative geometry of associative submanifolds is governed by gradient cycles in the base (see also \cite{Doan-Walpuski}\cite{Esfahani-Li}). A dimensional reduction is Calabi-Yau 3-folds with K3-fibrations, where one expects some correspondence between certain special Lagrangians and gradient cycles (see \cite{Chiu-Lin} for partial progress). 

\end{rmk}

    We now outline the proof of Thm.~\ref{thm:tropical}. In Section~\ref{sec: local models}, we introduce vertex and edge models for the special Lagrangians. Since the tropical curves under consideration are locally planar, these models are constructed as $T^{n-2}$-invariant preimages of holomorphic curves under hyperk\"ahler rotation. In Section~\ref{sec: matching data}, we introduce a slightly more general analogue of tropical curves, called matching data, which specifies the requirements for constructing approximate special Lagrangians by gluing these vertex and edge models.  In particular, in Prop.~\ref{prop: tropicaltomatchingdata} we show that tropical curves as defined in Def.~\ref{Def:tropicalcurve} give rise to matching data. Section~\ref{sec: perturbation} is the main analytic part of the paper. In Thm.~\ref{thm:main}, we prove that approximate solutions arising from matching data can be perturbed to special Lagrangians, provided the edge lengths are sufficiently large. A crucial technical ingredient is Prop.~\ref{prop:dstar}, where we construct a right inverse to the linearized operator $d^*$ acting on closed $1$-forms, with uniform bounds. This is proved using a parametrix construction, in which a partition of unity is used to remove local obstructions to solvability. The drawback of this method is that the resulting $1$-forms are not exact in general, namely the perturbation is not necessarily Hamiltonian. \newline

    \noindent{\bf Acknowledgments.}
     S. Chiu was supported by the NSF grant DMS-
1928930, when part of the work was performed at SLMath (formerly MSRI). Y. Li was supported by the Royal Society URF, and would like to thank Jeff Hicks for his interest. Y. Lin was supported by NSF grant DMS-2204109, the Travel Support for Mathematicians and a Simons Fellowship from the Simons Foundation.

	\section{Local models} \label{sec: local models}

	In this section, we explain the local models needed for the gluing constructions of special Lagrangians. Let $B= H^1(T^n,\R)\simeq \R^n$  be equipped with the natural integral affine structure, and $ g_{ij} dx^i dx^j$ is a given Euclidean metric on $B$.  The cotangent bundle $T^*T^n$ is equipped with the Euclidean K\"ahler structure
	\[
	\omega=\sum d\theta_i\wedge d\mu_i, \quad \Omega= \bigwedge_1^n (d\theta_i+ \sqrt{-1} g_{ij} d\mu_j), \quad g=g_{ij} d\mu_id\mu_j+ g^{ij} d\theta_id\theta_j.
	\]
	There is a projection map $\pi: T^*T^n\to B$.

	\subsection{Tropical curve}

	We recall some basic definitions of tropical curves (with edge weights equal to one).

	\begin{Def}\label{Def:tropicalcurve}
		A \emph{tropical curve} $(G,h)$ on $B$ is a connected graph $G=(V,E)$ with finite vertex set $V$ and edge set $E$, and a continuous map $h:G\rightarrow B$, such that \begin{enumerate}
			\item For every $e\in E$, $h(e)$ is affine with a nonzero rational slope. There are two types of edges: the internal edges have finite lengths, and the external edges are unbounded rays extending to infinity.

			\item (\emph{Balancing condition}) For each $v\in V$ and $e\in E$ adjacent to $v$, let $f_e$ denote the primitive tangent vector of $h(e)$ pointing outward from $h(v)$. Then 
			\begin{align*}
				\sum_{\text{e adjacent to v}} f_e=0.
			\end{align*} 
		\end{enumerate}
		Furthermore, we say that a tropical curve $(G,h)$ is called \emph{locally planar} if every vertex $v$ has valency at least $3$, and $\text{span}\{e : e \mbox{ adjacent to }v\}$ is a rank-two lattice.               
	\end{Def}

	For brevity, we will sometimes abbreviate $h(v), h(e)$ as $v, e$, As a caveat, the same $h(e)$ can be the image of several edges in the combinatorial graph $G$.
	

	\begin{Def}  Given tropical curve $(G,h)$ and a vertex $v\in V$, its localized tropical curve $(G_v,h_v)$ is defined as follows: $G_v$ is a tree consisting of a single vertex $v$ and unbounded edges correspond to edges in $G$ adjacent to $v$. Define $h_v(v)=h(v)$. For each edge $e$, then $h_v(e)$ is the ray which emanates from $h(v)$ and contains $h(e)$.

	\end{Def}

	\subsection{Edge model}\label{sec: edge}

	The local models for the edges $e$ are cylindrical special Lagrangians $L$ diffeomorphic to $T^{n-1}\times \R$, which are invariant under a given primitive subtorus $T^{n-1}\subset T^n$.

	Let $f_1,\ldots f_{n-1}\in H_{n-1}(T^n)$ be a $\Z$-basis of generators for the subtorus $T^{n-1}$, which induce moment maps $\mu_{f_i}$ by $d \mu_{f_i}=\omega(f_i, )$. Then $f_e=\bigwedge_1^{n-1} f_i \in H_{n-1}(T^n,\mathbb{Z})\simeq H^1(T^n,\mathbb{Z})$ specifies a rational direction in the base $B$, as well as the 1-form $d\theta_e(-)= \operatorname{Re}\Omega(f_1,\ldots f_{n-1},-)$, where $\theta_e$ is a circle coordinate on $T^n/T^{n-1}$. The Lagrangian condition for $L$ is equivalent to
	\[
	\mu_{f_i}= \text{const}, \quad i=1,\ldots n-1,
	\] 
	In other words, the projection of the special Lagrangian in $B$ is contained in an affine line parallel to $f_e$. The condition $\text{Im}(e^{-i\hat{\theta}}\Omega)=0$ translates into
	\[
	\operatorname{Im} (  e^{-i\hat{\theta}}\Omega) (f_1,\ldots f_{n-1}, )=  \operatorname{Im} (  e^{-i\hat{\theta}} ( d\theta_e+ \sqrt{-1} g_{ej} d\mu_j ))=0. 
	\]
	By assumption $\hat{\theta}\in (0,\pi)$, so the angle coordinate $\theta_e$ depends affine linearly on the moment coordinates:
	\begin{equation}\label{phaseonedge}
		\hat{\theta}_e:=  \theta_e- \cot \hat{\theta} g_{ej}\mu_j= \text{const} \in  T^n/T^{n-1}.
	\end{equation}

	\subsection{Vertex model}

	The local model for the vertices is invariant under a primitive subtorus $T^{n-2}\subseteq T^n$, and is diffeomorphic to a product of $T^{n-2}$ and a surface in $T^*T^2$. After suitable coordinate changes, without loss the subtorus is generated by $\partial_{\theta_3},\ldots \partial_{\theta_n}$. 
	The K\"ahler quotient $T^*T^n // T^{n-2}$ is $T^*T^2$ with 
	\[
	\begin{cases}
		\omega_{red}= d\theta_1\wedge d\mu_1+ d\theta_2\wedge d\mu_2,
		\\
		\Omega_{red}= \frac{1}{ \sqrt{ g_{11} g_{22}- g_{12}^2 }  } (d\theta_1+ \sqrt{-1} (g_{11} d\mu_1+ g_{12} d\mu_2) )\wedge (d\theta_2+ \sqrt{-1} (g_{21} d\mu_1+ g_{22} d\mu_2).
	\end{cases}
	\]
	Notice that  $\big(\omega_{red}, \text{Re}(e^{-i\hat{\theta} } \Omega_{red}),   \text{Im}(e^{-i\hat{\theta} } \Omega_{red})\big) $ forms a hyperk\"ahler triple. Then $L$ is a $T^{n-2}$-invariant special Lagrangian, iff it is contained in a fibre of the moment map $(\mu_3,\ldots \mu_n)$, and its image $\bar{L}$ in the K\"ahler quotient is a special Lagrangian:
	\[
	\begin{cases}
		\omega_{red}= 0,
		\\
		\text{Im}(e^{-i\hat{\theta} } \Omega_{red})=0.
	\end{cases}
	\]
	This is equivalent to a complex curve in the symplectic reduction with the new complex structure determined by the holomorphic volume form $$\text{Im}(e^{-i\hat{\theta} } \Omega_{red})+ \sqrt{-1}\omega_{red}. $$

	After some linear algebra, we find the new complex structure is identified with $(\C^*)^2$ with complex coordinates
	\begin{align}\label{eq: HK rel}
		\begin{cases}
			z_1= \exp \left( - \frac{ \sqrt{ g_{11} g_{22}- g_{12}^2 }     }{ \sin \hat{\theta} } \mu_2+   \sqrt{-1} ( -\cot \hat{\theta} ( g_{11}\mu_1+ g_{12}\mu_2 )+ \theta_1   )      \right)
			\\
			z_2=\exp \left(  \frac{ \sqrt{ g_{11} g_{22}- g_{12}^2 }     }{ \sin \hat{\theta} } \mu_1+   \sqrt{-1} ( -\cot \hat{\theta} ( g_{21}\mu_1+ g_{22}\mu_2 )+ \theta_2   )      \right).
		\end{cases}
	\end{align}
	The projection map to the base can be identified as 
	\[
	\text{Log}:(\C^*)^2\to \R^2.
	\]

	\begin{eg}
		After hyperk\"ahler rotation, the edge model cylinder becomes $z_1^az_2^b=\alpha$ for some $\alpha\in \mathbb{C}$, and its image in $\R^2$ is the affine line $ax+by=\log{|\alpha|}$.
	\end{eg}

	\subsubsection{Algebraic examples}

	We now discuss the examples from toric geometry.
	Let $N$ be a rank two lattice and $M$ be the dual lattice of $N$. Write 
	\[
	N_{\mathbb{R}}=N\otimes \mathbb{R},\quad M_{\mathbb{R}}=M\otimes \mathbb{R}, \quad N_{\mathbb{C}}=N\otimes \mathbb{C},\quad M_{\mathbb{C}}=M\otimes \mathbb{C}
	\] 
	We identify  $(\mathbb{C}^*)^2$ as the torus in $N_\C$. Each $m\in M$ parametrizes a monomial $z^m$. Let $f=\sum_m a_m z^m$ be a Zariski generic Laurent polynomial with Newton polytope $\Delta\subseteq M$, where
	\begin{align*}
		\Delta=\{m\in M_{\mathbb{R}}|\langle n_i,m\rangle \leq \nu_i\}
	\end{align*} for some primitive $n_i\in N$ and $\nu_i\in \mathbb{Z}$.
	Its zero loci $\mathcal{C}_f= \text{Zero}(f)$ defines a \emph{holomorphic curve} in $(\mathbb{C}^*)^2$, which we assume to be smooth, and $\mbox{Log}(\mathcal{C}_f)$ is the thickening of a \emph{tropical curve}. Up to hyperk\"ahler rotation, we can regard $\mathcal{C}_f$ as a special Lagrangian $\bar{L}\subset T^*T^2$, hence we get a model special Lagrangian $L\subset T^*T^n$.

	Let $m_0, m_1= m_0+ m_F, \ldots, m_l= m_0+ l m_F$ be the lattice points on a facet 
	\begin{align*}
		F=\{m\in \Delta\cap M|\langle m, n_F\rangle=\nu \}\subseteq\partial \Delta,
	\end{align*}
	where $n_F\in N$ is an outward-pointing primitive vector. The unbounded edges of the tropical curve in the direction $n_F$ are in one-to-one correspondence with the roots $\{ \alpha_1,\ldots \alpha_l\}$ of the polynomial
	\begin{align*}
		\sum_{k=0}^l a_{m_k}z^{k m_F}=0,
	\end{align*} 
	Since $f$ is generic, the roots are \emph{distinct}. 
	The following exponential asymptote is helpful for gluing the local models. We take $m_F'$ such that $m_F, m_F'$ is a basis of $M$, and $\langle m_F, n_F\rangle=1$.

	\begin{lem}\label{lem: asymp cyl} 
		For $R\gg 1$, each component $\mathcal{C}_{f,i}$ of 
		\begin{align}
			\mathcal{C}_f\cap \text{Log}^{-1}\{x\in N_{\mathbb{R}}|\langle x, m_F' \rangle > R \},
		\end{align}
		along the end of $\mathcal{C}_f$ parallel to $n_F$,
		is an $O(e^{-C'R})$-graph over the asymptotic cylinder $z^{m_F}=\alpha_i$.

		Furthermore, after hyperk\"ahler rotation identification, the special Lagrangian $L$ is asymptotically cylindrical, and along the $\mathcal{C}_{f,i}$ end, it is the graph of an exact 1-form $dc_i$ over the cylinder  with $\| c_i\|_{C^{k,\alpha}}<Ce^{-C'R}$.
	\end{lem} 
	\begin{proof}
		We choose a suitable basis $e_1=m_F,e_2=-m_F'$ of $M\cong \mathbb{Z}^2$. By shifting $\Delta$ by a multiple of $m_F$, we can assume that $\nu=0$, while $\mathcal{C}_f$ is unchanged. Then $f$ can be written as 
		\begin{align*}
			f(z_1,z_2)=a_{m_0}\prod_{i=0}^{l} (z_1-\alpha_i)+\sum_{m\in F^c\cap \Delta}a_m z^m.
		\end{align*} 
		Notice that the power of $z_2$ in $z^m$ is strictly positive if $m\in F^c\cap \Delta$. Along the $n_F$ end, $\log |z_1|$ is bounded while $|z_2| \leq C e^{- R}$. For  small $z_2$, consider the Taylor expansion around $z_1=\alpha_i$.
		For $R\gg 1$, we obtain 
		\begin{align} \label{New Year eve}
			|f(\alpha_i,z_2)|&=\big|\sum_{m\notin F}a_mz^m\big|_{z_1=\alpha_k}\big|<Ce^{-R}, \notag \\
			|\partial_{z_1}f|&\geq \frac{1}{2}\big|a_{m_0}\prod_{i\neq  k}(\alpha_k-\alpha_i)\big|>C^{-1}.
		\end{align} 
		By the implicit function theorem, we can solve $z_1-\alpha_i$ as an analytic function of the exponentially small $z_2$, which implies the exponential asymptote.

		After a suitable hyperk\"ahler rotation, we obtain a special Lagrangian $L$. In a Weinstein neighbourhood, $L$ is the graph of a closed 1-form $\mathfrak{c}_i$ over the asymptotic cylinder. By the exponential decay of this 1-form, for any $\gamma\in H_1(T^{n-1})$ along the asymptotic end, the periods
		$|\int_{\gamma}\mathfrak{c}_i|\leq Ce^{-C'R}$ for any large $R$. Thus these periods are zero, so $\mathfrak{c}_i= dc_i$ is exact. Upon integration we find $c_i=O(e^{-C'R})$.
	\end{proof}

	For any asymptotically cylindrical special Lagrangian
	$L\subset T^*T^n$ whose asymptotic ends $e$ are modelled on the $T^{n-1}$-invariant cylinders as in section \ref{sec: edge}, we can define the primitive vector $f_e\in H_{n-1}(T^n)\simeq H^1(T^n)$, and the phase constant $\hat{\theta}_e:= \theta_e- \cot \hat{\theta} \sum_j g_{ej}\mu_j$ along the asymptotic cylinders.
	The balancing condition for tropical curves reflect homological relation that the $T^{n-1}$ cycles for all the ends sum to zero:
	\[
	\sum_{\text{external edges}}  f_e=0.
	\]
	The above $T^{n-2}$-invariant special Lagrangians $L$ also satisfy the \emph{phase balancing condition}. 
	
	\begin{lem}
		We have the phase balancing condition
		\[
		\sum_{\text{external e}} \hat{\theta}_e= N\pi \mod 2\pi \Z,
		\]
		where $N$ is the number of asymptotically cylindrical ends.
		
	\end{lem}

	\begin{proof}
		Along the ends parallel to the $n_F$ direction as above, up to unravelling the hyperk\"ahler rotation, we see $\hat{\theta}_e= \arg \alpha_i$, where $\alpha_1,\ldots \alpha_l$ are the roots. 
		By Vieta's formula we have 
		$\alpha_1\ldots \alpha_l= (-1)^l a_{m_l}/ a_{m_0}$, so
		\[
		\sum_{k=1}^l \mbox{arg}(\alpha_k)=\mbox{arg}(a_{m_l})-\mbox{arg}(a_{m_0} )+ l\pi  \mod 2\pi \Z.
		\]
		Summing over the facets of $\Delta$ gives the result.
	\end{proof}

	\begin{rmk}\label{rmk:wellcentred}
		Likewise for all the ends parallel to the same $f_e$, we have $\sum_1^l \log |\alpha_i |= |a_{m_l}/ a_{m_0}| $. We will say that $\mathcal{C}_f$ is \emph{well centred at the origin}, if $ |\alpha_i|=1  $ along all the ends. (In particular, $|a_m|$ are the same positive number for all vertices $m\in \Delta$.) By the phase balancing condition, the choice of $\arg \alpha_i$ then lies in a (translated) copy of $T^{N-1}$ inside $U(1)^N$, subject to the Zariski open condition that $\alpha_i$ are distinct, and $\mathcal{C}_f$ is smooth. The point is that the projection of all the asymptotic cylinders to $B$ passes through the origin.
	\end{rmk}

	\begin{eg}
		(Trivalent vertices/pair of pants)
		Consider surfaces with three asymptotically cylindrical ends specified by the primitive vectors $e_1, e_2, e_3\in H^1(T^2, \Z)$. Homology imposes the \emph{balancing condition} $e_1+e_2+e_3=0$. Up to a $SL(2,\Z)$ transformation, we can reduce to the standard form $e_1=(-1,0), e_2=(0,1), e_3=(1,-1)$. The algebraic curves with these prescribed ends are given by
		\[
		\{ a z_1+ b z_2=1 \}\subset (\C^*)^2, \quad a, b\in \C^*.
		\]
		The effect of multiplying $a, b$ by real numbers is to translate the pair of pants in the $\R^2$ direction, and the well centred condition means $|a|=|b|=1$. The effect of $\arg a, \arg b\in U(1)$ is to rotate in the $T^2$-direction. 
		The choice of $a, b$ determines the phase constants $\hat{\theta}_e$ for two of the ends, and the phase constant for the third end is determined by the `\emph{phase balancing}' condition.

		In terms of the special Lagrangian, phase balancing has the following intrinsic description. Along the three ends there are three copies of the asymptotically invariant $T^{n-1}$, and $H^1(T^n/T^{n-1}_{e_i})\simeq \Z$ has generators $e_1, e_2, e_3$ respectively. Due to the balancing condition $e_1+e_2+e_3=0$, there is a natural homomorphism
		\[
		T^n/T^{n-1}_{e_1}\times T^n/T^{n-1}_{e_2} \times T^n/T^{n-1}_{e_3}\to U(1).
		\]
		The three asymptotic phase constants are intrinsically valued in $T^n/T^{n-1}_{e_i}$, and we require their product to be equal to $-1\in U(1)$.
	\end{eg}

	\section{Approximate special Lagrangian}

	\subsection{Matching Data and approximate special Lagrangian}\label{sec: matching data}

	In this subsection, we introduce the notion of matching data, which can be viewed as a variant of tropical curves, and then we discuss examples of matching data. 
	\begin{Def}
		A matching datum is a graph $G=(V,E)$, which parametrizes the following data: 
		\begin{itemize}
			\item For each vertex $v\in V$, let  $T_v\cong T^{n-2}$ be a primitive subtorus of $T^n$, and $p_v\in \mbox{Lie}(T_v)^*$.

			\item Let $T^*T^n\sslash_{p_v} T_v$ be the K\"ahler quotient at the moment map value $p_v$, then $X\sslash_{p_v} T_v\cong (\mathbb{C}^*)^2$ by a hyperK\"ahler rotation with coordinates from \eqref{eq: HK rel}. Let $f_v$ be a Laurent polynomial on $(\mathbb{C}^*)^2$, such that its zero locus $\mathcal{C}_v$ is smooth. Section \ref{sec: local models} produces an asymptotically cylindrical special Lagrangian $L_v$ by lifting $\mathcal{C}_v$.

			We require the asymptotic cylindrical ends of $L_v$ are in bijection with the edges $e\in E$ adjacent to $v\in V$. We denote these $T^{n-1}$-invariant asymptotic half cylinders as $L_{v,e}\subset T^*T^n$.

			\item For each edge $e\in E$ between the vertices $v,v'\in V$, we require that the half cylinders $L_{v,e}$ and $L_{v',e}$ are contained in the same $T^{n-1}$-invariant special Lagrangian cylinder $C_e$, and the intersection $L_{v,e}\cap L_{v',e}$ is non-empty, and its projection to the base $B$ is an interval with length $l_e\gg 1$ bigger than some large constant.   
			
		\end{itemize} 
		
	\end{Def}

	\begin{rmk}
		Notice that in the definition of the matching data, a priori there may not be a natural map from $G$ to $B$. This is because the projection of the asymptotic cylinders of the vertex local model to $B$ may not coincide. 
	\end{rmk}

	We can construct an approximate special Lagrangian $L$ from a matching datum as follows. For each internal edge $e$, let $s\in [-l_e/2, l_e/2]$ be the coordinate on the image interval of the cylinder $C_e$. We take a cutoff function
	\begin{align*}
		\chi(s)=\begin{cases}  
			1 & \mbox{ if }s\leq -2 \\
			0 & \mbox{ if }s\geq -1.
		\end{cases}
	\end{align*}
	Let $v,v'$ be the two vertices of $e$.
	By Lemma \ref{lem: asymp cyl}, the asymptotic end $e$ of $L_v$ is  $\mbox{graph}(dc_{v,e})$ for an exponentially decaying function $c_{v,e}$ over $C_e$. We construct the Lagrangian $L$ to agree with $L_v$ in the region with $s\leq -2$ (which has distance $\sim l_e/2\gg 1$ to the core of $L_v$), and replace the $s\geq -2$ end with 
	$
	\mbox{graph}(d\big(c_{(v,e)}\chi(s))).
	$
	Likewise, we can glue $C_e$ to $L_{v'}$ on the region $1\leq s\leq 2$.

	This Lagrangian $L\subset T^*T^n$ is approximately special Lagrangian.

	\begin{lem} \label{lem: ansatz SLAG}
		Given a matching data, then $Err= *\mbox{Im}(e^{i\hat{\theta}}\Omega)|_L$ is supported on the region $s\in [-2,2]$, and $\norm{ Err}_{C^{k,\alpha}} \leq C e^{-C'l_e/2}$, where the constants depend on the geometry of $\mathcal{C}_v$, but not on $l_e$.

	\end{lem}

	\begin{lem}
		We have $\int_{L} \operatorname{Im}(e^{-i\hat\theta}\Omega) = 0$.
	\end{lem}
	
	\begin{proof}
		Since the integrand is locally supported, 
		we focus locally near a vertex model $L_v$. Let $\tilde L_v$ be $L_v$ glued to its asymptotic cylinder via the graph of $d\big(c_{(v,e)}\chi(s))$. We can smoothly interpolate $L_v$ to $\tilde L_v$ using the graph of 
		$d(  c_{(v,e)} (\chi(s) + t (1-\chi(s))   )$. Since $c_{v,e}$ has exponential decay, Stokes formula implies that $\int\operatorname{Im}(e^{-i\theta}\Omega)$ is constant under the interpolation, so 
		\begin{align*}
			\int_{\tilde L_v} \operatorname{Im}(e^{-i\theta}\Omega) = \int_{L_v} \operatorname{Im}(e^{-i\theta}\Omega) = 0.
		\end{align*}
		Summing over the local contributions give the result.
	\end{proof}

	\subsection{From tropical curves to matching data}

	There is a large supply of matching data from locally planar tropical curves.

	\begin{prop}\label{prop: tropicaltomatchingdata}
		Let $(G,h)$ be a locally planar tropical curve in $B$. Then there exists a family of matching data parametrized by $T\gg 1$ and an open dense set $U\subset U(1)^{|E|-|V|}$, such that 
		\begin{itemize}
			\item The subtori $T_v$ are independent of $T$, and $\mathcal{C}_v$ are independent of $T$ up to translation in $B$.

			\item  The projection of the asymptotic half cylinders $L_{v,e,T}$ to $B$, is the ray from $Th(v)$ along the direction of $h(e)$.
		\end{itemize}

	\end{prop}

	\begin{proof}
		For each vertex $v\in V$, let $(G_v,h_v) $ be the local tropical curve, whose edges $e$ determine outward pointing primitive directions $f_e$. The span of $f_e$ is a 2-dimensional lattice $N$, which determine $T_v\simeq T^{n-2}\subset T^n$ whose Lie algebra is annihilated by the $f_e$. The points $h(v)\in B$ determine $p_v\in \text{Lie}(T_v)^*$. We take a 1-parameter family of dilations $p_{v,T}=Tp_v$.

		The symplectic reduction $T^*T^n\sslash_{p_v} T_v$ is identified with the big torus in $(\C^*)^2\subset N_\C$.
		By  the balancing condition
		$
		\sum_e  f_e=0.
		$
		We collect together all the $f_e$ pointing in the same directions $n_i$, so $\sum l_i n_i=0$ counting the number of parallel edges $l_i$. Up to translation, there is a unique Newton polytope $\Delta_v \subset M= N^*$, whose facets have outward normals $n_i$ and lengths $l_i$.
		We consider \emph{well centred} $\mathcal{C}_f$ for Laurent polynomials $f=f_v'$ with Newton polytope $\Delta_v$ (see Remark \ref{rmk:wellcentred}). The choices of $f_v'$ are subject to the \emph{phase balancing condition}, and a Zariski open condition.

		Now the tropical curve data $h(v)\in B\simeq H^1(T^n, \R)$ determine a point $\bar{h}(v)$ in $N_\R\simeq H^1(T^n,\R)/H^1(T_v, \R)\simeq \R^2$. We  translate $\mathcal{C}_{f}$ by the amount $T\bar{h}_v\in N_\R$, which gives $\mathcal{C}_{v,T}$. Up to reversing the hyperk\"ahler rotation, we can then lift $\mathcal{C}_v$ to a special Lagrangian $L_{v,T}\subset T^*T^n$. Its asymptotic (half) cylinders $L_{v,e}$ are labelled by the edges $e$ adjacent to $v\in V$. Thanks to the well-centred condition, the projection of $L_{v,e}$ to $B$ is the ray starting from $h(v)\in B$ and pointing in the direction $f_e$.

		For each edge $e$ between $v,v'$, we need to ensure that $L_{v,e}$ and $L_{v',e}$ to be contained in the same cylinder $C_e$. The invariant $T^{n-1}$ direction is determined by $f_e$, and the projection of $L_{v,e}$ and $L_{v',e}$ to $B$ are contained in the same line determined by $Th(e)$. It suffices to ensure that the phase constants $\hat{\theta}_e$ match up. We need to solve the system of phase balancing constraint equations:
		\begin{equation}\label{eqn: phasebalancing}
			\sum_{ \text{e adjacent to v}  } \hat{\theta}_e= N_v\pi \mod 2\pi\Z,
		\end{equation}
		where $N_v$ is the number of edges $e$ emanating from $v$.

		We claim that the solution of (\ref{eqn: phasebalancing}) is parametrized by $U(1)^{|E|-|V|}$. For this, we notice that the phase constant of an external edge can always be solved in terms of the other edges, so we can remove vertices one by one and proceed by induction. Moreover, we can evaluate the global solutions $U(1)^{|E|-|V|}$ to the local solution set $U(1)^{N_v-1}$ for the phase balancing condition at any given $v\in V$. This evaluation map is surjective (eg. one can see this by arranging $v$ to be the last vertex to be removed). The Zariski closed condition for $\mathcal{C}_f$ to be singular, or for the roots to collide, defines a closed and nowhere dense subset of $U(1)^{|E|-|V|}$. The upshot is that we have an open dense subset of phase parameters $U\subset U(1)^{|E|-|V|}$, such that the well centred conditions, phase balancing, and smoothness conditions are all satisfied. 
		
		Finally, we notice that the length of the projection of $L_{v,e}\cap L_{v',e}$ is $T$ times the length of $h(e)$, which is arbitrarily large for $T\gg 1$.
	\end{proof}

	\section{Perturbing to special Lagrangian}\label{sec: perturbation}
	
	In this section, we prove the main technical theorem:

	\begin{thm}\label{thm:main}
	Given a family of matching data, such that the subtori $T_v$ and the geometry of the vertex models $\mathcal{C}_v$ are fixed. Let
	\[
		l_{min}= \min_{\text{internal edges e}} l_e,\quad 	l_{max}= \max_{\text{internal edges e}} l_e.
	\]
	Suppose $l_{max}/l_{min}\leq C$ is uniformly bounded, while $l_{min}\to +\infty$.  
	Then for $l_{min}$ sufficiently large, there is a special Lagrangian which is the graph of a $C^{k,\alpha}$-small closed 1-form over the preglued Lagrangian $L$.
	\end{thm}

By Prop. \ref{prop: tropicaltomatchingdata}, this immediately implies Thm. \ref{thm:tropical}.
The rest of the section is devoted to proving
	Theorem~\ref{thm:main}.

 By the definition of matching data, for every
	edge $e$ there is a $T^{n-1}$-invariant special
	Lagrangian cylinder $C_e$, and for each vertex $v$ we
	can find a $T^{n-2}$-invariant special Lagrangian model $L_v$ whose
	asymptotic cylinders are given by $C_e$, where $e$ are adjacent edges
	of $v$. Thus two adjacent vertices share a common asymptotic cylinder.
	
	Let $L$ be the preglued Lagrangian submanifold. We summarize the
	properties of $L$ as follows:
	
	\begin{itemize}
		
		\item Let $\ell_e$ denote the length of the finite cylinder in
		$C_e$. Then we choose the cylindrical coordinate
		$s \in [-\frac{\ell_e}{2},\frac{\ell_e}{2}]$ so that $s$ is
		increasing in the same direction of the orientation of $e$. The
		cylinder is glued to the vertex models $L_v, L_{v'}$ in the regions
		$s\in [-2, -1]$ and
		$[1, 2]$, respectively.

		\item Let $t$ denote the cylindrical coordinate of $L_v$ along the end
		asymptotic to $C_e$. The region $t \in [l_e/2-2, l_e/2-1]$ in $L_v$ is glued to
		$C_e$.
		
		\item The initial error
		$Err = *\operatorname{Im}(e^{-i\hat\theta}\Omega)|_{L}$ is
		supported in the gluing region, with $Err(t) \le Ce^{-C' l_e/2}$ for
		some constant $C'>  0$.
	\end{itemize}

	To perturb $L$ to a genuine special Lagrangian submanifold, we will
	follow the implicit function theorem approach. To study the
	linearized problem, we adopt the parametrix method, which requires
	studying the invertibility of the linearized operators on vertex and
	cylinder models.
	
	To this end, we fix a smooth partition of unity $\{ \chi_v\}$ with $\sum_v \chi_v=1$ and $0\leq \chi_v\leq 1$ as follows:
	
	\begin{itemize}
		
		\item For vertices $v$, we define $U_v$ to be the region with $t\leq3l_e/4$ along the internal edges, and $t$ unconstrained for the infinite edges. Thus the union of $U_v$ cover the entire $L$.

		\item  For each vertex $v$, we require
		$ \chi_v $ is supported in $U_v$, such that $\chi_v=1$ in the region with $t\leq l_e/2-1$ for all the internal edges, and $\chi_v=0$ for $t\geq l_e/2+1$. In particular, $\chi_v=1$ along the external edges.

	\end{itemize}

	\subsection{Linear analysis}

	For functions or 1-forms $f \in C^{k,\alpha}_{loc}(L)$, we define the weighted norms
	\begin{align*}
		\|f\|_{C^{k,\alpha}_\delta(L)}
		= \sum_{v} \| \chi_v f\|_{C^{k,\alpha}_\delta(U_v)}.
	\end{align*}
	The main result is

	\begin{prop}\label{prop:dstar}
		The following holds for $-C'< \delta < 0$ sufficiently small, and $l_{min}\gg 1$ sufficiently large. Let $f \in C^{k,\alpha}_{\delta}(L)$
		satisfy $\int_{L} f = 0$. Then there exists a closed 1-form $\beta\in C^{k+1,\alpha}_\delta(L)$ solving
		$d^*\beta= f$ such that
		\[
		\beta=\beta_1+ \sum_v c_v d\chi_v,\quad \|\beta_1 \|_{C^{k+1,\alpha}_\delta(L)} + \sum_v |c_v|  \le C\|f\|_{C^{k,\alpha}_\delta(L)} \]
		for some uniform constant. Moreover, we can require $\beta$ to be exact inside every vertex region $U_v$.
	\end{prop}

	\begin{rmk}
		A rather subtle issue is that if we restrict to \emph{globally exact} 1-forms, we can find a unique solution $\beta\in C^{k+1,\alpha}_\delta(L)$ solving $d^* \beta =f$, but then we cannot in general guarantee the uniform estimate in $T$. We note also that the global exactness condition is equivalent to the local exactness when the graph $\Gamma$ is a \emph{tree}, but not when $\Gamma$ contains loops.
	\end{rmk}

	We first begin with a simple general lemma.
	
	\begin{lem}
		\label{lem:zero_average}
		Let $(L,g)$ be a fixed
		asymptotically cylindrical manifold  with $N$ ends, and suppose $\delta<0$ is small enough. Let
		$f \in C^{k,\alpha}_\delta(L)$ such that $\int_L f=0$. Then there exists a unique exact 1-forms $\beta\in C^{k+1}_\delta(L)$ solving $d^*\beta= f$, such that
		$ 
		\|\beta\|_{C^{k+1,\alpha}_\delta(L)}  \le C\|f\|_{C^{k,\alpha}_\delta(L)}  $. 
		
	\end{lem}
	
	\begin{proof}
		Suppose $-\delta>0$ is smaller than the exponential convergence rate of the cylindrical end, and the  first indicial root of the Laplacian on the asymptotic cylinders. By Lockhart-McOwen theory, for such $\delta < 0$, the Laplacian on functions
		$\Delta: C^{k+2,\alpha}_{\delta}(L) \to C^{k,\alpha}_{\delta}(L)$ is
		Fredholm with index $-N$.

		Let $c_1, c_2, \ldots, c_N$ be a partition of unity on $L$ such that
		$c_i=1$ on the $i$-th end and is equal to zero on the other ends.  
		By standard index theory, the Laplace
		operator
		\begin{align*}
			\Delta:
			C^{k+2,\alpha}_{\delta}(L) \oplus \operatorname{Span}_{\mathbb{R}}\{c_1,c_2,\ldots, c_N\}
			\to
			C^{k,\alpha}_{\delta}(L)
		\end{align*}
		has index $0$. To determine its kernel, suppose that
		$ 
		\Delta u = 0
		$ 
		for some $u = u_0 + \sum_i b_i c_i$, where
		$u_0 \in C^{k+2,\alpha}_{\delta}(L)$ and $b_i \in \mathbb{R}$. Then we have
		\begin{align*}
			0 = \int_L u\Delta u  = - \int_L |\nabla u|^2.
		\end{align*}
		Thus $u_0 = 0$ and $b_1 = b_2 = \cdots = b_N$. It follows that the
		kernel has dimension $1$, and so does the cokernel. The condition
		$\int_L f = 0$ implies that $f$ lies in the image of
		$\Delta$. Now modulo constant functions,
		\[
		C^{k+2,\alpha}_{\delta}(L) \oplus \operatorname{Span}_{\mathbb{R}}\{c_1,c_2,\ldots, c_N\}  \]
		correspond precisely to the exact 1-forms on $L$ in $C^{k+1,\alpha}_\delta$, hence the result.
	\end{proof}

	The main problem is how to make the estimate uniform for $L$ when $T$ is large. We will address this by a parametrix construction. Let $f\in C^{k,\alpha}_\delta(L)$ satisfy $\int_{L} f=0$.

	We need to remove the vertex 
	integrals $b_v = \int_{L} \chi_v f$, as they obstruct the
	solvability in the local model. For each vertex
	$v$, we will find real numbers $x_v$ such that the new function
	\begin{equation}
		f_2 = f + \sum_{v} x_{v}\Delta\chi_{v} 
	\end{equation}
	satisfies
	$
	\int_{L} \chi_v f_2 = 0
	$
	for all vertices $v$. By the Green formula
	\[
	A_{vv'} := \int_{L} \nabla \chi_{v} \cdot \nabla \chi_{v'}= -\int_{L} \chi_v \Lap \chi_{v'} ,
	\]
	this boils down to solving the following linear system:
	\begin{align*}
		\sum_{v'} A_{vv'}x_{v'} = b_v \:\:\: \text{ for all vertex } v.
	\end{align*}

	\begin{lem}
		The $|V|\times |V|$ coefficient matrix
		$A_{vv'} $ is positive
		semi-definite, whose $1$-dimensional kernel is generated by $(1,1, \cdots, 1)$.
	\end{lem}
	
	\begin{proof}
		Let $y=(y_v) \in \mathbb{R}^{V}$. We have
		\begin{align*}
			y^TAy = \int_{L} |\sum_v y_v\nabla \chi_v|^2 \ge 0.
		\end{align*}
		Equality holds if and only if $\nabla(\sum y_v\chi_v) = 0$, i.e.
		$\sum y_v\chi_v = c$ is constant. Since $\chi_v$ is a partition of
		unity, equality holds if and only if $y_v = c$ for all $v$.
	\end{proof}
	
	Since
	\begin{align*}
		\sum_v b_v 
		= \sum_v \int_{L} \chi_v f 
		= \int_{L} f
		= 0,
	\end{align*}
	the vector $(b_v)$ is orthogonal to
	$\operatorname{ker} A^T =
	\operatorname{span}_{\mathbb{R}}\{(1,1,\ldots,1)\}$, so we can find $x=(x_v)$ solving
	the linear system $Ax=b$, with
	\begin{equation}
		|x_v|\le C\sum_{v'} |b_{v'}| \leq C\|f\|_{C^{k,\alpha}_\delta(L)},\quad \forall v,
	\end{equation}
	for a uniform constant $C$.

	We recall $\int_{L} \chi_v f_2=0$. For each vertex $v$, we apply Lemma~\ref{lem:zero_average} to find an exact 1-form $\beta_v\in C^{k,\alpha}_\delta(\tilde{L}_v)$ on the (slightly perturbed) vertex model spaces $\tilde{L}_v$ which agrees with the local part of $L$, such that 
	$d^* \beta_v= \chi_v f_2$.

	To construct an approximate solution globally, we need to cut off the solution. On the asymptotically cylindrical model $\tilde{L}_v$, in the end along any internal edge direction $e$, we can write
	$
	\beta_v= du_{v,e}$ for a unique function $ u_{v,e}\in C^{k+2,\alpha}_\delta.
	$
	For the vertex $v$, and any cylindrical end of $\tilde{L}_v$ along an internal edge direction $e$, we take another cutoff function $\tilde{\chi}_{v,e}$, supported in $t\geq 2l_e/3$, and equals one on $t\geq 2l_e/3+1$. Since $l_e\gg 1$, by construction the support of $\chi_v$ is far separated from the support of $\tilde{\chi}_{v,e}$, by a distance $\sim l_e/6$. We modify $\beta_v$ to $\beta_v'=\beta_v- \sum_e d(\tilde{\chi}_{v,e} u_{v,e})$. This is still an exact 1-form on $\tilde{L}_v$ in the weighted space $C^{k,\alpha}_\delta$, but now $\beta_v'=0$ for $t\geq 2l_e/3+1$, so we can regard $\beta_v'$ as a 1-form on $L$.

	\begin{rmk}
		A subtlety here is that if we write the exact 1-form $\beta_v$ as $du_v$ for some global function $u_v$ on $\tilde{L}_v$, then $u_v= u_{v,e}+ \text{const}$, where the constants are generally different along the different ends. If we regard $u_v$ as a local function on an open subset of $L$, it is unclear how to \emph{consistently extend $u_v$ to a global function} on $L$, without introducing some additional cutoff error $du_v$ somewhere else in $L$. This error is \emph{not small}. The upshot is that to gain uniform estimates, we should sacrifice global exactness, but only preserve exactness on the local regions $U_v$.
		
		
	\end{rmk}

	We can now write down the parametrix
	$
	Pf
	= \sum_v \beta_v' + \sum_v x_vd\chi_v.
	$
	Thus 
	\[
	\begin{split}
		d^*(Pf)-f = & \sum_v d^*\beta_v'+ \sum_v  x_vd^*d \chi_v-f
		\\
		= & \sum_v d^*\beta_v'- \sum_v x_v\Lap \chi_v- f
		\\
		= & \sum_v (d^*\beta_v'- \chi_v f_2).
	\end{split}
	\]
	But $d^*\beta_v'- \chi_v f_2$ is due to the cutoff error supported on $2l_e/3\leq t\leq 2l_e/3+1$. It is bounded by
	\[
	\norm{d^*\beta_v'- \chi_v f_2}_{C^{k,\alpha}}\leq C e^{2\delta l_e/3}\norm{\beta_v}_{C^{k+1,\alpha}_\delta(\tilde{L}_v) } \leq  C e^{2\delta l_e/3} \norm{ f   }_{C^{k,\alpha}_\delta (L) } .
	\]
	Now on $L$, this error is supported at distance $\sim l_e/3$ from the other vertex with edge $e$, so 
	\[
	\norm{d^*\beta_v'- \chi_v f_2}_{C^{k,\alpha}_\delta(L)  } \leq C e^{\delta l_e/3} \norm{ f   }_{C^{k,\alpha}_\delta (L) }.
	\]
	The upshot is that $\norm{ d^* Pf-f }\leq C\sum_{v,e} e^{\delta l_e/3} \norm{f} $.

	Define $R = I- d^* P $. When $l_{min}=\min_e l_e\gg 1$, we can
	ensure that the operator norm $\|R\| \leq C \max_{v,e} e^{\delta l_e/3} \ll \frac{1}{100}$. Then 
	\begin{align*}
		\tilde{P} = P\sum_{k=0}^\infty R^k 
	\end{align*}
	is a right inverse of $d^*$ with uniform estimate.
	This completes the proof of Proposition~\ref{prop:dstar}.

	\subsection{Nonlinear iteration}
	
	We now set up the implicit function theorem argument.

	Since $L$ has uniform geometry (i.e. $L$ has
	uniform second fundamental form bound and injective radius lower
	bound), by the Weinstein Lagrangian neighbourhood theorem there exists some uniform $r_0 > 0$,
such that for each $T > 0$ there exists a
	symplectomorphism
	\begin{align*}
		\Psi: B_{r_0}(0_{L}) \subset T^*L \longrightarrow U_L \subset X, \:\:\: \Psi(0_{L}) = L.
	\end{align*}
	We seek a closed 1-form $\beta$ on $L$ which is exact on the vertex regions $U_v$, such that the perturbed Lagrangian
	$
	L^\beta := \Psi(\operatorname{graph}(\beta))
	$
	is special Lagrangian with desired estimates.
	Define the nonlinear differential operator $F(\beta)$ which sends the closed 1-forms $\beta$ with $\norm{\beta}_{C^{k+1,\alpha}_\delta}\leq r_0$, to the $n$-form $\operatorname{Im}(e^{-i\hat\theta}\Omega)|_{L^\beta}$ pulled back to $L$ via the graph $\beta$, which is then identified as a function on $L$.
	Since $L^\beta$ and $L$ have the same asymptote at infinity up to exponentially decaying error, by Stokes formula and Lemma \ref{lem:zero_average},
	\begin{equation}\label{eqn:integralzero}
		\int_{L^\beta} \operatorname{Im}(e^{-i\hat\theta}\Omega) =0,
	\end{equation}
	so crucially $\int_{L} F(\beta)=0 $.
	The goal then is to find $\beta$ such that $F(\beta) = 0$.

	Recall that the initial error $Err$ is supported in $t\sim l_e/2$, and $\norm{Err}_{C^{k,\alpha}} \leq C e^{-C'l_e/2}$. 
	We suppose $-C'/2\leq  \delta<0$ is small enough so that the linear theory applies. 
	The Taylor expansion of $F(\beta)$ gives
	\begin{align*}
		F(\beta) =  Err - d^* \beta +  Q(\beta),
	\end{align*}
	where for $\norm{\beta}_{C^{k+1,\alpha} }\ll 1$, we have the quadratic estimate for local $C^{k,\alpha}$-norms,
	\[
	\norm{Q(\beta)}_{ C^{k,\alpha} } \leq C\norm{\beta}_{C^{k,\alpha} }(\norm{\theta_T-\hat{\theta} }_{C^{k,\alpha}}+ \norm{\beta}_{C^{k,\alpha}}) \leq C\norm{\beta}_{C^{k,\alpha} }(  e^{-C' l_{min}/2} + \norm{\beta}_{C^{k,\alpha}}) 
	\]
	and for any such $\beta_1, \beta_2$, 
	\[
	\norm{Q(\beta_1)- Q(\beta_2) }_{ C^{k,\alpha}} \leq C (  e^{-C' l_{min}/2}  +\norm{ \beta_1 }_{C^{k,\alpha} }+ \norm{\beta_2}_{C^{k,\alpha} }) \norm{\beta_1-\beta_2}_{ C^{k,\alpha} }).
	\]

	For $f\in C^{k,\alpha}_\delta(L)$ with $\int_{L} f=0$, we have the right inverse $\tilde{P}$ to $d^*$. We seek a solution to $F(\beta)=0$ in the form $\beta= \tilde{P} f$, namely
	\[
	Err- f + Q(\tilde{P} f)=0.
	\]
	In other words, we want to find a fixed point of the nonlinear map
	\begin{align*}
		N(f) = Err +  Q(\tilde{P} f).
	\end{align*}
	Using the uniform norm estimate in Prop. \ref{prop:dstar}, 
	\[
	\tilde{P} f= \beta'+ \sum_v c_v d\chi_v,\quad    \norm{\beta'}_{C^{k+1,\alpha}_\delta} + \sum_v |c_v| \leq C \norm{f}_{C^{k,\alpha}_\delta}. 
	\]
	Thus  for any functions $f_1, f_2$ with small $\norm{f_i}_{C^{k,\alpha}_\delta} \leq C e^{-C' l_{min}/4} $,  we have
	\[
	\begin{split}
	&	\|N(f_1) - N(f_2)\|_{C^{k,\alpha}_\delta(L)} 
		\\
	\leq 	&  C e^{-\delta l_{max}/2}(  e^{-C'l_{min}/2} +\norm{f_1}_{C^{k,\alpha}_\delta(L)}+  \norm{f_2}_{C^{k,\alpha}_\delta(L)} )  \|f_1-f_2\|_{C^{k,\alpha}_\delta(L)}.
	\end{split}
	\]
	Notice that the bad exponentially large factor $e^{-\delta l_{max}/2}$ shows up because the $\sum_v c_v d\chi_v$ term does not have the same decay as $\beta'$.
	We now pick $\delta$ so that $|\delta| < C' l_{min}/ (10 l_{max})$. 
	The result then follows from  Banach iteration.
	This proves Thm. \ref{thm:main}.

\end{document}